\documentclass{amsart}

\usepackage[all]{xy}   
\usepackage{hyperref, graphicx} 
\usepackage[usenames,dvipsnames]{xcolor}

\theoremstyle{plain}
\newtheorem{lem}{Lemma}[section]
\newtheorem{cor}[lem]{Corollary}
\newtheorem{prop}[lem]{Proposition}
\newtheorem{thm}[lem]{Theorem}

\theoremstyle{definition}
\newtheorem{ex}[lem]{Example}
\newtheorem{rem}[lem]{Remark}
\newtheorem{dfn}[lem]{Definition}

\newcommand{\Z}{\mathbb{Z}}               
\newcommand{\Q}{\mathbb{Q}}              
\newcommand{\R}{\mathbb{R}}              

\newcommand{\II}{\mathcal{I}}  

\begin{document}

\title[Twisted foldings]{Twisted quadratic foldings of root systems}

\author[M.~Lanini]{Martina Lanini}
\address[Martina Lanini]{Dipartimento di Matematica, 
Universit\`a di Roma Tor Vergata,
Via della Ricerca Scientifica,
00133 Rome,
Italy}
\email{lanini@mat.uniroma2.it}
\urladdr{https://sites.google.com/site/martinalanini5/home}

\author[K.~Zainoulline]{Kirill Zainoulline}
\address[Kirill Zainoulline]{Department of Mathematics and Statistics, University of Ottawa, 585 King Edward Street, Ottawa, ON, K1N 6N5, Canada}
\email{kirill@uottawa.ca}
\urladdr{http://mysite.science.uottawa.ca/kzaynull/}


\begin{abstract} 
In the present paper we study twisted foldings of root systems which generalize usual involutive foldings corresponding to automorphisms of Dynkin diagrams.
Our motivating example is the Lusztig projection of the root system of type $E_8$ onto the subring of icosians of the quaternion algebra which gives the
root system of type $H_4$. 

Using moment graph techniques for any such folding we construct a map at the equivariant cohomology level. We show that this map commutes with characteristic classes and Borel maps. We also introduce and study its
restrictions to the usual cohomology of projective homogeneous varieties, to group cohomology 
and to their virtual analogues for finite reflection groups.\end{abstract}

\maketitle

\section{Introduction}

Consider a root system $\Phi$ together with its geometric realization in $\R^N$, that is
we look at $\Phi$ as a subset of vectors in $\R^N$ which is closed under
reflection operators $r_v$ for each $v\in \Phi$. 
A basic example of such a realization is given by the vectors
$$\Phi=\{\alpha_{ij}=e_i-e_j \in \R^{2n} \mid i\neq j,\; i,j=1\ldots 2n,\; n\ge 1\},$$ 
where $\{e_1,\ldots,e_{2n}\}$ is the standard
basis of $\R^{2n}$. This corresponds to the root system of Dynkin type $A_{2n-1}$. 
Note that $\R^{2n}$ admits an involutive symplectic linear operator $\tau(e_i)=-e_{2n-i+1}$ 
which preserves $\Phi$. Taking averages over orbits in $\Phi$ under $\tau$ one obtains a new subset of vectors 
\[\Phi_{\tau}=\{\tfrac{1}{2}(\alpha_{ij}+\tau(\alpha_{ij}))\mid \alpha_{ij}\in \Phi\}
\]
which turns out to be a geometric realization of the root system of Dynkin type~$C_n$.
There are other similar examples of root systems (such as $D_{n+1}$ and $E_6$) and involutive operators $\tau$ induced by automorphisms of the respective Dynkin diagrams.
The procedure of passing from $\Phi$ to $\Phi_\tau$ by taking averages via $\tau$
is called folding of a root system (see Steinberg~\cite{St67}).

In the present paper we are weakening the assumption $\tau^2=id$ by considering an arbitrary  linear automorphism $\mathcal{T}$ of $\R^N$
that satisfies a separable quadratic equation \[p(\mathcal{T})=\mathcal{T}^2-c_1\mathcal{T}+c_2=0,\quad c_1,c_2\in \R,\]
hence, dealing with the so called twisted (quadratic) foldings. 

Our motivating example is the celebrated projection of the root system of type $E_8$ onto the subset of icosians of the quaternion algebra which realizes a finite non-crystallographic root system of type $H_4$.
This twisted folding has been first noticed by Lusztig \cite{Lu83} while investigating square integrable representations of semisimple $p$-adic groups, and then studied by many authors, among them
by Moody-Patera \cite{MP} in the context of quasicrystals, by Shcherbak~\cite{Sh88} from the point of view of singularity theory, by Dyer~\cite{Dy} using the language of embeddings of root systems.
More precisely, in this case one considers an operator $\mathcal{T}$ that satisfies the quadratic equation $x^2-x-1=0$ of the golden section 
and then takes the projection onto an eigenspace. 

The key point of our work is to use the twisted foldings and moment graph techniques
to construct maps from equivariant cohomology with coefficients in a ring $k$ of flag varieties  
to their virtual analogue for finite Coxeter groups.
This follows Sorgel's philosophy to use virtual geometry to investigate combinatorics of Coxeter groups~\cite{Soe}.

As the first step we formalize the notion of a twisted folding over $k$ by looking at a root system $\Phi$ as a subset of  the Weil restriction $R_{l/k}\mathcal{U}$ of a free module $\mathcal{U}$ over $l$
where $l=k(x)/(p(x))$ is a quadratic separable algebra over an integral domain $k$. 
The operator $\mathcal{T}$ is the multiplication by a root $\tau$ of $p$ that preserves the root lattice of $\Phi$ and partitions the subset of simple roots of 
$\Phi$. 
We define the $\tau$-twisted folding  (Definition~\ref{dfn:taufolding}) to be the 
projection $\Phi\to \Phi_\tau$ of $\Phi$ to the respective 
$\tau$-eigenspace of~$\mathcal{T}$.  

Then we consider the moment graphs $\mathcal{G}$ and $\mathcal{G}_\tau$
associated to the root systems $\Phi$ and $\Phi_\tau$.
As in Fiebig~\cite{F1}, we introduce the structure algebras 
$\mathcal{Z}(\mathcal{G})$ and $\mathcal{Z}(\mathcal{G}_\tau)$
which play the role of (virtual) equivariant cohomology and their augmentations $\overline{\mathcal{Z}}(\mathcal{G})$ and $\overline{\mathcal{Z}}(\mathcal{G}_\tau)$ (the usual cohomology).
Our main result (Theorem~\ref{thm:mainthm}) says that the $\tau$-twisted folding $\Phi\to \Phi_\tau$ induces canonical maps 
\[
\iota^* \colon \mathcal{Z}(\mathcal{G}) \to \mathcal{Z}(\mathcal{G}_\tau)\quad\text{ and }\quad
\bar\iota^* \colon \overline{\mathcal{Z}}(\mathcal{G}) \to \overline{\mathcal{Z}}(\mathcal{G}_\tau)\]
between the (augmented) structure algebras which commute with forgetful maps.

We show that these map commute with characteristic classes and Borel maps 
(see Lemma~\ref{lem:charm}). We also provide several criteria for its surjectivity in positive degrees (see Corollary~\ref{cor:surjc} and Remark~\ref{rem:surj}). As an example, for the Lusztig folding $E_8\to H_4$
we obtain a ring homomorphism from the equivariant Chow ring of a certain $E_8$-flag variety to the much smaller structure algebra
associated to the Coxeter group $H_4$ (see Example~\ref{ex:mapH})
\[ \iota^*\colon CH_T(E_8/P_{D_6}; k) \longrightarrow \mathcal{Z}(H_4/P_{H_3};l),\quad k=\Z[\tfrac{1}{5}],\; l=k(\tau),\; \tau=\tfrac{1}{2}(1+\sqrt{5}).\]  
Observe that this map is surjective for $l=\mathbb{Q}(\tau)$ 
and while the left hand side is a free module of rank 30240, the right hand side has only rank 120.

At the end of the paper we focus on the study of quotient of the augmented structure algebra
modulo an ideal generated by characteristic classes. 
We call it the reduced structure algebra and denote it by $\widehat{\mathcal{Z}}(\mathcal{G})$. 
Observe that the reduced structure algebra (e.g. $\widehat{\mathcal{Z}}(\mathcal{G}_\tau)$)
can be defined for any finite reflection group (not necessarily crystallographic).

In the crystallographic case $\widehat{\mathcal{Z}}(\mathcal{G})$  coincides with the Chow ring $CH(G;k)$
of the adjoint form of the respective algebraic group $G$. 
The latter is a celebrated geometric invariant of $G$ which plays an important role in the invariant and representation theory of $G$.
For history and summary of computations of $CH(G;k)$ when $k=\Z$ or $k$ is a finite field of modular characteristic we refer to Kac~\cite{Ka85}.
For recent computations of $CH(G;k)$, $k=\Z$ for adjoint forms of exceptional groups we refer to Duan-Zhao~\cite{DZ15}.

As an application, for the Lusztig foldings we obtain ring homomorphisms between reduced structure algebras
\begin{align*}
\hat\iota^*_{E_8} &\colon CH(E_8;k)=\widehat{Z}(E_8;k) \longrightarrow \widehat{Z}(H_4;l), \\
\hat\iota^*_{D_6} &\colon CH(PGO_{12};k)=\widehat{Z}(D_6;k) \longrightarrow \widehat{Z}(H_3;l), \\
\hat\iota^*_{A_4} &\colon CH(PGL_5;k)=\widehat{Z}(A_4;k) \longrightarrow \widehat{Z}(H_2;l).
\end{align*}
Using these maps together with several natural assumptions (surjectivity and restrictions on $p$-exceptional degrees) we provide conjectural answers for the reduced structure algebras 
of the finite reflection groups $H_2$, $H_3$ and $H_4$.

Let us mention that in recent years foldings coming from finite group actions on Dynkin diagrams have been investigated from a Soergel bimodule perspective by Lusztig in an updated version of \cite{Lu14} and by Elias \cite{El16}, to categorify quasi-split Hecke algebras with unequal parameters. It would be interesting also to relate twisted foldings to Hecke algebras with unequal parameters via sheaves on moment graphs, in the spirit of  \cite{La12}.

\paragraph{\bf Acknowledgements} M.L. acknowledges the MIUR Excellence Department Project awarded to the Department of Mathematics, University of Rome Tor Vergata, CUP E83C18000100006. K.Z. was partially supported by the NSERC Discovery grant RGPIN-2015-04469, Canada.

\section{Eigenspace decomposition under quadratic base change}
 
Let $k$ be a subring of $\R$. 
Consider a  quadratic separable algebra $l$ over $k$ \cite[Ch.III, \S4]{Kn91} so that $l=k[x]/(p(x))$, where $p(x)$ is a monic quadratic separable polynomial over $k$.
Let $\tau$ and $\sigma$ denote the roots of this polynomial in $l$, so that
\[p(x)=x^2-c_1x+c_2=(x-\tau)(x-\sigma),\]
where $c_1=\tau+\sigma\in k$ and $c_2=\tau\sigma \in k$.
Then $l=k(\tau)$ as a free module of rank 2 over $k$ with the $\tau$-basis $\{1,\tau\}$. 
So each element of $l$ can be written uniquely as $y+\tau y'$, $y,y'\in k$ which we denote as $(y,y')$.
The ring multiplication in $l$ considered over $k$ is then given by
 \[
(y,y')\cdot(\hat y,\hat y')=\big(y\hat y - c_2 y'\hat y',y\hat y'+y'\hat y+c_1 y'\hat y'\big).
 \]

It is well-known that such an algebra $l$ is either isomorphic to a product $k\times k$ (split case)
or it is a nontrivial twisted form of $k\times k$ (non-split case). 

We always assume that both the free term $c_2$ (hence, $\tau$ and $\sigma$) and the discriminant $D=(\tau-\sigma)^2=c_1^2-4c_2$ (hence, $\tau-\sigma$)
are invertible in $k$.

The split case corresponds (up to an isomorphism) to the polynomial $p(x)=x^2-1$ ($c_1=0$ and $c_2=-1$). 
Whenever we deal with the split case we always fix the roots in $l$ as $\tau=(0,1)$ and $\sigma=(0,-1)$ and assume that $2$ is invertible in $k$.

In the non-split case we additionally assume that $l$ is also a subring of $\R$ so that $\tau$ is a real root and $D>0$.

\begin{ex}

If we write elements of $l=k\times k$ in terms of the standard basis as pairs $[z,z']$, $z,z'\in k$ then the unit  in $l$ is $1=[1,1]=(1,0)$ and 
the roots are $\tau=[1,-1]$, $\sigma=[-1,1]$. For an element of $l$ we have 
\[(y,y')= [y+y',y-y']\quad\text{ and }\quad [y,y']=\tfrac{1}{2}(y+y',y-y'),\;  y,y'\in k.\] 
For the ring multiplication in $l$ we obtain
\[(y,y')\cdot (\hat y,\hat y')=\big(y\hat y + y'\hat y',y\hat y'+y'\hat y\big)\quad\text{ and }\quad
[y,y']\cdot [\hat y, \hat y']=[y\hat y,y'\hat y'].
\]
\end{ex} 

\begin{ex}\label{ex:golden} Consider a case where $k=\Q$ and $p(x)=x^2-x-1$ ($c_1=1$, $c_2=-1$) is an irreducible quadratic polynomial.
Then $l$ is the unique intermediate quadratic Galois extension of the cyclotomic extension $\Q(\zeta_5)/\Q$. 
We can also take $k=\Z$ and $l$ to be the ring of integers $\Z(\tau)$ of $l$. 
Observe that both examples are non-split cases and we have $\tau=\tfrac{1}{2}(1+\sqrt{5})$ and $\sigma=\tfrac{1}{2}(1-\sqrt{5})$ 
(note that $\tau$ is the so called golden section).
\end{ex}

Consider a free module $\mathcal{U}$ of rank $n$ over $l$. Each element in $\mathcal{U}$ is an $n$-tuple $(x_1,\ldots,x_n)$ where $x_i\in l$.
 We denote by $U$ the module $\mathcal{U}$ viewed as a free module over $k$ by expressing each $x_i$ in the basis $\{1,\tau\}$ (observe that $l$ is a 2-dimensional vector space over $k$, so $U$ has dimension $2n$ over $k$). We will call such $U$ the Weil restriction of $\mathcal{U}$ with respect to the ring extension $l/k$.
Each element $u\in U$ can be written as 
 \begin{equation}\label{eq:Uelement}u=((y_1,y_1'),(y_2,y_2'),\ldots,(y_n,y_n'))\end{equation} where $x_i=y_i+\tau y_i'$ and $y_i,y_i'\in k$. 
 The element $u\in U$ viewed as an element in $\mathcal{U}$ will be denoted as $u_l$.
 We identify the submodule $((y_1,y_1'),(0,0),\ldots)$ with $l$, and the submodule $((y_1,0),(0,0),\ldots)$ with $k$.

 We introduce a symmetric $k$-bilinear form $B_\tau (\cdot,\cdot)$ on $U$ by taking
 the standard $l$-bilinear dot product $(-\cdot-)$ in $\mathcal{U}$
 and then the $k$-linear projection $pr_\tau\colon l\to k$, $(y,y')\mapsto y$, i.e., \[B_\tau(u, \hat u):=pr_\tau(u_l\cdot \hat u_l),\quad u,\hat u\in U.\] 
 We call it the $\tau$-form on $U$. Similarly, we can define the form $B_\sigma(\cdot,\cdot)$. 
By definition, we have
 \[
 B_\tau\big( (\ldots,(y_i,y_i'),\ldots),  (\ldots,(\hat y_i,\hat y_i'),\ldots) \big)=\sum_{i=1}^n \big(y_i\hat y_i -c_2 y_i'\hat y_i'\big)
 \]
 which implies the following
 \begin{lem}\label{lem:dotpr}
 The $\tau$-form (resp. $\sigma$-form) on $U$ coincides with the standard $k$-bilinear dot product on $U$
 if and only if $c_2=\tau\sigma=-1$.
 
 In particular it holds in the split case and in the non-split case of Example~\ref{ex:golden}.
 \end{lem}
 
 We now introduce a key object of the present section
\begin{dfn} Consider an $l$-linear automorphism of $\mathcal{U}$ given by multiplication by~$\tau \in l$, i.e, 
\[
(\ldots,x_i,\ldots)\mapsto (\ldots,\tau\cdot x_i,\ldots),\quad x_i\in l.
\]
Viewed as a $k$-linear automorphism of $U$ it maps
\[
(\ldots, (y,y')_i,\ldots) \mapsto (\ldots, (-c_2y', y+c_1y')_i,\ldots).
\]
We denote it by $\mathcal{T}$ and call it the $\tau$-operator. By definition we have for each $u\in U$
\[(\mathcal{T}u)_l=\tau u_l\in\mathcal{U}. \]
\end{dfn}

Consider the base change $U_l=U\otimes_k l$ and the respective $\tau$-operator 
$\mathcal{T}_l\colon U_l\to U_l$ extended linearly over $l$ 
(here we consider the coordinates $y_i$, $y_i'$ in \eqref{eq:Uelement} as taken from~$l$).
(Observe that $\mathcal{T}_l(v)$ is not necessarily $\tau v$ for $v\in U_l$.)
By definition, the operator 
$\mathcal{T}_l$ has eigenvalues $\tau$ and $\sigma$ in $l$ and the corresponding eigenspaces can be described as follows.

By definition we have
\begin{align*}
\mathcal{T}_l(x,x')=\tau(x,x')\;&\Longleftrightarrow\; (-c_2 x', x+c_1x')=(\tau x,\tau x') \\
&\Longleftrightarrow\; -c_2 x'=\tau x\text{ and }x+c_1x'=\tau x'\\
&\Longleftrightarrow\; -c_2 x'=\tau x\text{ and }x=-\sigma x'.
\end{align*}
Since $\tau x=\tau(-\sigma x')=-c_2 x'$, the eigenspace for $\tau$ is given by
\[
U_\tau=\{(\ldots,(-\sigma x',x'),\ldots)\mid x'\in l\}.
\]
Similarly, the eigenspace for $\sigma$ is given by
\[
U_\sigma=\{(\ldots,(-\tau x',x'),\ldots)\mid x'\in l\}.
\]

\begin{lem}\label{lem:eigensp} There is a direct sum decomposition of $l$-modules
$U_l=U_\tau\oplus U_\sigma$ given by
\[
(\ldots,(x,x')_i,\ldots)=(\ldots,(-\sigma a,a)_i,\ldots)+(\ldots,(-\tau b,b)_i,\ldots),
\]
where
$a=\tfrac{1}{\tau-\sigma}(x+\tau x')$, $b=\tfrac{1}{\sigma-\tau}(x+\sigma x')$ and $x,x'\in l$.
\end{lem}

\begin{proof}
We have
\begin{align*}
x=-\sigma a -\tau b,\; x'=a+b\;&\Longleftrightarrow\; x=-\sigma a-\tau(x'-a),\; b=x'-a\\
&\Longleftrightarrow\; a=\tfrac{1}{\tau-\sigma}(x+\tau x'),\; b=\tfrac{1}{\sigma-\tau}(x+\sigma x').\qedhere
\end{align*}
\end{proof}

We will extensively use the following description of eigenspaces of the $\tau$-operator
\begin{prop}\label{ex:coreig} We have 
\[U_\tau=(\mathcal{T}-\sigma I_l)(U)\;\text{ and }\;U_\sigma=(\mathcal{T}-\tau I_l)(U),\] 
where $U$ is considered as a $k$-submodule of $U_l$.

Moreover, for any $z\in U$ we have $z=z_\tau+z_\sigma$, where
\[
z_\tau=\tfrac{1}{\tau-\sigma}(\mathcal{T}-\sigma I_l)(z)\in U_\tau\;\text{ and }\; z_\sigma=\tfrac{1}{\sigma-\tau}(\mathcal{T}-\tau I_l)(z)\in U_\sigma.
\]
\end{prop}

\begin{proof} We consider only the $l$-linear operator $\mathcal{T}-\sigma I_l$. Similar arguments can be applied to $\mathcal{T}-\tau I_l$.

For any $(\ldots, (y,y'),\ldots)\in U\subset U_l$ (here $y,y'\in k\subset l$) we have
\begin{align*}
(\mathcal{T}-\sigma I_l)(\ldots, (y,y'),\ldots)&=(\ldots, (-c_2y'-\sigma y,y+c_1y'-\sigma y'),\ldots) \\
&=(\ldots,(-\sigma(y+\tau y'),y+\tau y'),\ldots)) \in U_l,
\end{align*}
where $y+\tau y'$ represents an arbitrary element of $l$ in the basis $\{1,\tau\}$. Set $z=(\ldots, (y,y'),\ldots)$
and apply Lemma~\ref{lem:eigensp}. 

We then obtain $z_\tau=(\ldots,(-\sigma a,a),\ldots)$, where $a=\tfrac{1}{\tau-\sigma}(y+\tau y')\in l$ and the result follows.
\end{proof}

\begin{ex} In the split case, the $\tau$-operator over $l$ is the switch involution
\[
\mathcal{T}_l\colon (\ldots, (x,x')_i,\ldots) \mapsto (\ldots, (x', x)_i, \ldots), \quad x_i,x_i'\in l.
\]
Moreover, we have
\[
a=\tfrac{1}{\tau-\sigma}(x+\tau x')=\tfrac{1}{-2\sigma}(x+\tau x')=\tfrac{1}{2}(x'+\tau x),
\]
\[
b=\tfrac{1}{\sigma-\tau}(x+\sigma x')=\tfrac{1}{-2\tau}(x+\sigma x')=\tfrac{1}{2}(x'-\tau x)
\]
So the eigenspace decomposition of the $i$-th coordinate over $l$ is given by
\begin{align*}
(x,x') &=(-\sigma a,a)+(-\tau b,b)\\
&=\tfrac{1}{2}(x+\tau x', x'+\tau x)+\tfrac{1}{2}(x-\tau x',x'-\tau x) 
\end{align*}
\end{ex}

\begin{lem}\label{lem:adjoint}
The $\tau$-operator $\mathcal{T}$ is adjoint, i.e., $B_\tau(z,\mathcal{T}(z'))=B_\tau(\mathcal{T}(z),z')$.
\end{lem}

\begin{proof} By definition of the dot-product we have
\[
B_\tau(z,\mathcal{T}(z'))=pr_\tau(z\cdot_l \tau z')=pr_\tau(\tau z\cdot_l z')=B_\tau(\mathcal{T}(z),z').\qedhere
\]
\end{proof}

Consider the $\tau$-form $B_\tau(\cdot,\cdot)$ on $U$ and extend it to $U_l$ over $l$ by linearity.
\begin{lem} 
The eigenspace decomposition $U_l=U_\tau\oplus U_\sigma$ is an orthogonal sum decomposition with respect to $B_\tau$ over $l$.
\end{lem}

\begin{proof}
Suppose $u\in U_\tau$ and $u'\in U_\sigma$. 
Then by Proposition~\ref{ex:coreig} $u=(\mathcal{T}-\sigma I_l)(z)$ and $u'=(\mathcal{T}-\tau I_l)(z')$ for some $z,z'\in U \subset U_l$.
So we obtain
\begin{align*}
B_\tau \big((\mathcal{T}-\sigma I_l)(z),(\mathcal{T}-\tau I_l)(z')\big) &= \\
B_\tau(\mathcal{T}(z),\mathcal{T}(z'))-\tau B_\tau(\mathcal{T}(z),z')- \sigma B_\tau(z, \mathcal{T}(z'))+\tau\sigma B_\tau (z,z') &= \\
pr_\tau\big(\tau^2(z_l\cdot z_l')\big)-\tau pr_\tau \big( \tau (z_l\cdot z_l') \big)-\sigma pr_\tau \big( \tau (z_l\cdot z_l') \big)+\tau\sigma pr_\tau(z_l\cdot z_l') &= \\
pr_\tau\big(\tau^2(z_l\cdot z_l')\big)- c_1 pr_\tau \big( \tau (z_l\cdot z_l') \big)+ c_2 pr_\tau(z_l\cdot z_l') &= \\
pr_\tau(p(\tau)(z_l\cdot z_l')) &=0.\qedhere
\end{align*}
\end{proof}

Consider the composite 
$\pi_\tau\colon U\to U_l\to U_\tau$
where the first map is the base change and the second one is the orthogonal projection
with respect to $B_\tau$. Following Proposition~\ref{ex:coreig} it is given by 
\[\pi_\tau(u)=\tfrac{1}{\tau-\sigma}(\mathcal{T}-\sigma I_l)(u),\; u\in U.\]
If we view $U$ as the module $\mathcal{U}$ over $l$, then
$\pi_\tau$ induces an $l$-linear isomorphism 
\[\pi_\tau\colon \mathcal{U} \stackrel{\simeq}\to U_\tau,\qquad u_l=(\ldots,x_i,\ldots)\mapsto \tfrac{1}{\tau-\sigma}(\ldots,(-\sigma x_i,x_i),\ldots).
\]
Since $B_\tau$ extended over $l$ is defined by
\[
B_\tau\big((\ldots,(x,x')_i,\ldots),(\ldots,(\hat x,\hat x')_i,\ldots)\big)=\sum_{i} (x\hat x -c_2 x'\hat x'),
\]
we obtain
\begin{align*}
B_\tau(\pi_\tau(u),\pi_\tau(\hat u)) &= \tfrac{1}{(\tau-\sigma)^2} B_\tau\big( (\ldots,(-\sigma x_i,x_i),\ldots),(\ldots,(-\sigma \hat x_i,\hat x_i),\ldots)\big)\\
&= \tfrac{\sigma^2-c_2}{D}\sum_i x_i\hat x_i=\tfrac{c_1\sigma-2c_2}{D}(u_l\cdot \hat u_l).
\end{align*}
Observe that in the split case ($c_1=0$ and $c_2=-1$) we have \[B_\tau(\pi_\tau(u),\pi_\tau(\hat u))=\tfrac{1}{2}(u_l\cdot u_l).\]

Finally, we will need the following
\begin{lem}\label{lem:normrat}
Let $u\in U$ be such that $(u_l\cdot u_l)\in k$. Then
\begin{itemize}
\item[(1)]
$B_\tau(u,\mathcal{T}u)=0$ and 
\item[(2)] $B_\tau(\mathcal{T}u,\mathcal{T}u)=-c_2 B_\tau(u,u)=-c_2 (u_l\cdot u_l)$.
\item[(3)] Moreover, if $(u_l\cdot u_l)\neq 0$, then $u\neq \mathcal{T}u$.
\end{itemize}
\end{lem}

\begin{proof}
(1) By the definition of $B_\tau$ we have 
\[
B_\tau(u,\mathcal{T}u) =pr_\tau(u_l\cdot (\mathcal{T}u)_l) = pr_\tau(u_l\cdot \tau u_l) = pr_\tau(\tau(u_l\cdot u_l)).
\]
Since $(u_l\cdot u_l)\in k$, the latter equals $0$ by the definition of $pr_\tau$.

(3) We have
\begin{align*}
B_\tau(\mathcal{T}u,\mathcal{T}u) &=pr_\tau((\mathcal{T}u)_l\cdot (\mathcal{T}u)_l) =pr_\tau(\tau u_l\cdot \tau u_l) \\
&=pr_\tau(\tau^2(u_l\cdot u_l)) =pr_\tau((c_1\tau-c_2)(u_l\cdot u_l))\\
&\stackrel{(*)}=-c_2 pr_\tau (u_l\cdot u_l)=-c_2 B_\tau(u,u).
\end{align*}
Again in ($*$) we used the definition of $pr_\tau$ and the fact that $(u_l\cdot u_l)\in k$.

(3) follows from (1) and (2). 
\end{proof}

\begin{dfn}
A nonzero element $u\in U$ such that $(u_l\cdot u_l)\in k$ will be called $\tau$-rational.
\end{dfn}

Observe that in the non-split case if $u$ is $\tau$-rational , then $\mathcal{T}u$ is not $\tau$-rational.

\begin{rem}
Consider a canonical involution $\rho$ of $l$ which switches the roots $\tau$ and $\sigma$.
It can be extended to an involution on $U$ by
\[
\rho(\ldots,x_i,\ldots)=(\ldots,\rho(x_i),\ldots), \quad x_i\in l.
\]
Observe that in the split case, it maps \[(\ldots, (y,y')_i,\ldots)\mapsto (\ldots, (y,-y')_i,\ldots ).\]
Note also that $\rho$ has eigenvalues $\pm 1$ over $k$ and, hence, it induces an eigenspace decomposition over $k$
\[
(\ldots,(y,y')_i,\ldots)=(\ldots,(y,0)_i,\ldots)+(\ldots,(0,y')_i,\ldots).
\]

All the results of the present section respect $\rho$, meaning that any statement/proof for $\tau$ holds for $\sigma$ automatically
by applying the involution $\rho$, e.g., we have \[B_\tau(u,\hat u)=B_\sigma(\rho(u),\rho(\hat u))\quad \text{ for }u,\hat u\in U.\]
\end{rem}

\begin{rem}\label{rem:splitproj}
Consider the $k$-linear projection $p\colon l\to k$ given by $\tau:=1$. By definition we have $p(y,y')=y+y'$, $p[y,y']=y$ and
the induced map $p\colon U_l\to U$ is given by $(\ldots,x_i,\ldots)\mapsto (\ldots,p(x_i),\ldots)$, $x_i\in l$.

In the split case, $p$ projects the eigenspace decomposition $U_l=U_\tau\oplus U_\sigma$ of the switch involution $\mathcal{T}$ onto the eigenspace decomposition over $k$
\begin{equation}\label{eq:spliteigen}U=p(U_\tau)\oplus p(U_\sigma)\end{equation}
\[(y,y')=\tfrac{1}{2}(y+y', y'+y)+\tfrac{1}{2}(y-y',y'-y),\]
where $p(U_\tau)=U_1$ is the submodule of $\mathcal{T}$-invariants and $p(U_\sigma)=U_{-1}$ is the submodule of $\mathcal{T}$-skewinvariants.

Whenever we discuss the eigenspace decomposition in the split case we will always mean the decomposition~\eqref{eq:spliteigen} over $k$
obtained by applying the projection $p$.
\end{rem}

\section{Folded representations}

Following \cite[Exp.~XXI, \S1.1]{SGA} we define a  {\em root datum} to be an embedding
$\Phi\hookrightarrow \Lambda^\vee$, $\alpha\mapsto \alpha^\vee$, 
of a non-empty finite subset $\Phi$ of a free finitely generated abelian group $\Lambda$ into its $\Z$-linear dual $\Lambda^\vee$
such that 
\begin{itemize}
\item
$\Phi \cap 2\Phi=\emptyset$, 
$\alpha^\vee(\alpha)=2$ for all $\alpha \in \Phi$, and
\item
$\beta - \alpha^\vee(\beta)\alpha \in \Phi$ and 
$\beta^\vee - \beta^\vee(\alpha)\alpha^\vee \in \Phi^\vee$ 
for all $\alpha,\beta \in \Phi$, where $\Phi^\vee$ denotes the image of
$\Phi$ in $\Lambda^\vee$.
\end{itemize}
The elements of $\Phi$ (resp. $\Phi^\vee$) are called roots (resp. coroots). 
The sublattice of $\Lambda$ generated by $\Phi$ 
is called the {\em root lattice} and is denoted by $\Lambda_r$. The root lattice $\Lambda_r$ admits a basis
$\Delta=\{\alpha_i\}$ 
such that 
each $\alpha \in \Phi$ is a linear combination of $\alpha_i$'s with 
either all positive or all negative coefficients. 
So the set $\Phi$ splits into two disjoint subsets 
$\Phi = \Phi^+ \amalg \Phi^-$, where
$\Phi^+$ (resp. $\Phi^-$) is called the set of positive (resp. negative)
roots.  The roots $\alpha_i \in \Delta$ are called {\em simple roots}.

Consider a  root datum $\Phi\hookrightarrow \Lambda^\vee$ where we fix a subset of simple roots~$\Delta$. 
Let $U$ be a free module of rank $2n$ ($2n\ge |\Delta|$) over $k$ together with a symmetric positive-definite bilinear form $B(-,-)$ over $k\subset \R$; 
and
let $\Phi$ be a subset in $U$ 
so that \[\alpha^\vee(u)=2B(\alpha,u)/B(\alpha,\alpha),\quad \alpha\in \Phi,\; u\in U.\]
Let $W$ denote the respective Weyl group generated by reflections \[r_\alpha(u):=u-\alpha^\vee(u) \alpha,\quad u\in U,\;\alpha\in \Delta.\]
We will refer to $U$ as a representation of the root datum.

\begin{dfn}\label{dfn:foldedreps} We say that the representation $U$ is a folded representation of the root datum $\Phi\hookrightarrow\Lambda^\vee$
if 
 \begin{itemize}
 \item[(1)] There exists a quadratic separable algebra $l/k$ and a free module $\mathcal{U}$ of rank $n$ over $l$ such that 
 $U$ is the Weil restriction of $\mathcal{U}$ with respect to the ring extension $l/k$. 
 
 \noindent (Following the previous section we then fix roots $\tau, \sigma \in l$, the $\tau$-operator $\mathcal{T}\colon U\to U$
 and the $\tau$-form $B_\tau$.)
  \item[(2)] The standard bilinear form $B(-,-)$ coincides with the respective $\tau$-form $B_\tau(\cdot,\cdot)$ over $k$.
 \item[(3)] The $\tau$-operator $\mathcal{T}$   
 stabilizes the root lattice $\Lambda_r\subset U$, i.e., $\mathcal{T}$ is a $\Z$-linear automorphism of $\Lambda_r$.
 \item[(4)] The set of simple roots $\Delta$ of $\Phi$ partitions as a disjoint union
 \[
 \Delta=\Delta^{\mathcal{T}}\amalg \Delta_{rat}\amalg \mathcal{T}(\Delta_{rat}),
 \]
 where $\Delta^{\mathcal{T}}$ is the subset of all $\tau$-invariant simple roots (i.e.,  
 $\alpha=\mathcal{T}\alpha$) and $\Delta_{rat}$ is a subset consisting of $\tau$-rational simple roots.
\end{itemize}
\end{dfn}

\begin{rem}\label{rem:foldreps}
Since $B_\tau$ is a standard bilinear form, by Lemma~\ref{lem:dotpr} $c_2=\tau\sigma=-1$. We will always assume $\tau>0$. Lemma~\ref{lem:normrat} guaranties
that for all $\alpha\in \Delta_{rat}$ the roots $\alpha$ and $\mathcal{T}\alpha$ are orthogonal and have the same length. 
The condition (3) and 
the assumption that $\Delta_{rat}$ and $\mathcal{T}(\Delta_{rat})$ are disjoint imply that $|\Delta_{rat}|=|\mathcal{T}(\Delta_{rat})|$.
\end{rem}

\begin{rem}
We will see later that in a non-split case we always have $\Delta^{\mathcal{T}}=\emptyset$ and all roots in $\mathcal{T}(\Delta_{rat})$ are non rational. But in a split case both subsets $\Delta_{rat}$
and $\mathcal{T}(\Delta_{rat})$ consist of rational roots.
\end{rem}

From now on we assume that $U$ is a folded representation of a root datum $\Phi\hookrightarrow \Lambda^\vee$.

\underline{Given a $\tau$-rational simple root} $\alpha$ consider the $k$-linear reflections $r_\alpha$ and $r_{\mathcal{T}\alpha}$.
For simplicity we set $c=B_\tau(\alpha,\alpha)=B_\tau(\mathcal{T}\alpha,\mathcal{T}\alpha)$ 
($\alpha$ and $\mathcal{T}\alpha$ have the same length).
Observe that they commute with each other since $\mathcal{T}\alpha$ is orthogonal to $\alpha$.

We then have
 \begin{align*}
 r_{\mathcal{T}\alpha}r_\alpha u &=r_{\mathcal{T}\alpha}(u- \alpha^\vee(u) \alpha) \\
&= u- \alpha^\vee(u) \alpha - (\mathcal{T}\alpha)^\vee\big( u-\alpha^\vee(u) \alpha \big) \mathcal{T}\alpha \\
&= u-\alpha^\vee(u)\alpha - (\mathcal{T}\alpha)^\vee(u)\mathcal{T}\alpha.
 \end{align*}

\begin{lem}\label{lem:commuting} The $k$-linear operator $r_{\mathcal{T} \alpha}r_\alpha$ 
commutes with $\mathcal{T}$.
\end{lem}

\begin{proof}  For any $u\in U$ we have
\begin{align*}
c \mathcal{T}(r_{\mathcal{T} \alpha}r_\alpha u) &= \mathcal{T}(cu-2B_\tau(u,\alpha)\alpha - 2B_\tau(u,\mathcal{T}\alpha)\mathcal{T}\alpha)\\
&=\mathcal{T}(cu) - 2B_\tau(u, \alpha) \mathcal{T} \alpha - 2B_\tau (u, \mathcal{T} \alpha) \mathcal{T}^2 \alpha \\
&=\mathcal{T}(cu) - 2B_\tau(u, \alpha) \mathcal{T} \alpha - 2B_\tau (u, \mathcal{T} \alpha) (c_1\mathcal{T}+id) \alpha\\
& =\mathcal{T}(cu) - 2B_\tau(u,\mathcal{T}\alpha) \alpha - 2B_\tau (u, c_1\mathcal{T}\alpha+\alpha)\mathcal{T}\alpha \\
& =\mathcal{T}(cu) - 2B_\tau(u,\mathcal{T}\alpha) \alpha - 2B_\tau (u, \mathcal{T}^2\alpha)\mathcal{T}\alpha \\
& =\mathcal{T}(cu) - 2B_\tau(\mathcal{T} u,\alpha) \alpha - 2B_\tau (\mathcal{T} u, \mathcal{T}\alpha)\mathcal{T}\alpha \\ &= cr_{\mathcal{T} \alpha}r_\alpha (\mathcal{T} u),
\end{align*}
(here we used the fact that $\mathcal{T}$ is adjoint (see Lemma~\ref{lem:adjoint}) and the relation $\mathcal{T}^2=c_1\mathcal{T}+id$).
The result then follows since $k$ is a domain.
\end{proof}

Let $R\colon U_l\to U_l$ denote the $k$-linear operator $r_{\mathcal{T}\alpha}r_\alpha$ extended by linearity over $l$.
\begin{cor} 
The operator $R$ preserves the eigenspaces of $\mathcal{T}_l$, i.e., $R(U_\tau)=U_\tau$ and $R(U_\sigma)=U_\sigma$.
\end{cor}

Consider the orthogonal projection
 \[
 \pi_\tau \colon U \to U_l \to U_\tau,\quad
 \pi_\tau (u)= \tfrac{1}{\tau-\sigma}(\mathcal{T}-\sigma I)(u).
 \]
 Set $ \bar{\alpha}=\pi_\tau(\alpha)=\tfrac{1}{\tau-\sigma}(\mathcal{T}\alpha-\sigma \alpha)\in U_\tau$
 to be the image of $\alpha$.

 \begin{lem}  The $l$-linear operator $R$ satisfies the following
 \begin{itemize}
 \item[(1)] $R(\bar\alpha)=-\bar\alpha$,
 \item[(2)] $R(u)=u$ for all $u\in U_\tau$ such that $B_\tau(u,\bar\alpha)=0$.
 \end{itemize}
 \end{lem}
 
\begin{proof}
(1) By linearity we have
$R (\bar\alpha) =\tfrac{1}{\tau-\sigma}  R(\mathcal{T}\alpha-\sigma \alpha) =\tfrac{1}{\tau-\sigma}\left(R(\mathcal{T}\alpha)-\sigma R(\alpha)\right)$. Since 
$r_{\mathcal{T}\alpha}(\alpha)=\alpha$ and $r_{\alpha}(\mathcal{T}\alpha)=\mathcal{T}\alpha$ (here we used orthogonality), 
we get \[R(\mathcal{T}\alpha)=r_{\mathcal{T}\alpha}(\mathcal{T}\alpha)=-\mathcal{T}\alpha \;\text{ and }\;
R(\alpha)=r_{\mathcal{T}\alpha}(-\alpha)=-\alpha.\]  So we obtain 
$R (\bar\alpha)=\tfrac{1}{\tau-\sigma}(-\mathcal{T}\alpha+\sigma \alpha)=-\bar\alpha$.

(2) Suppose $B_\tau(u,\bar\alpha)=0$, that is $B_\tau(u,\mathcal{T}\alpha-\sigma \alpha)=0$ $\Longleftrightarrow$ $B_\tau(u,\mathcal{T}\alpha)=\sigma B_\tau(u,\alpha)$.
We then obtain
\begin{align*}
cR(u)&=cu-2B_\tau (u, \alpha) \alpha - 2B_\tau(u, \mathcal{T} \alpha) \mathcal{T} \alpha\\
&=cu-2B_\tau(u, \alpha) \alpha - 2\sigma B_\tau(u, \alpha) \mathcal{T} \alpha\\
&=cu-2B_\tau(u,\alpha)(\alpha+\sigma \mathcal{T}\alpha).
\end{align*}
To conclude, we notice that $\alpha+\sigma \mathcal{T}\alpha=0$, since $\tau(\alpha+\sigma\mathcal{T}\alpha)=\tau \alpha-\mathcal{T}\alpha\in U_\sigma$ (cf. Proposition~\eqref{ex:coreig}).
But $R$ stabilizes $U_\tau$ and, hence, $\alpha+\sigma \mathcal{T}\alpha\in U_\sigma\cap U_\tau=\{0\}$.
\end{proof}

\begin{cor} The operator $R$ is an $l$-linear reflection if restricted to $U_\tau$ with respect to $\bar\alpha$ and the $l$-form $B_\tau$. 
So we will denote it by $R_{\bar \alpha}$.
\end{cor}

\underline{Given a $\tau$-invariant simple root} $\alpha=(\ldots,(y,y')_i,\ldots)$, observe that
\[\mathcal{T}\alpha=(\ldots,(y',y+c_1y')_i,\ldots)=(\ldots,(y,y')_i,\ldots)=\alpha\] implies $c_1=0$
and $y=y'$. So we are in the split case and $\alpha=\bar\alpha\in U_\tau$. As in the rational case we set $c=B_\tau(\alpha,\alpha)$.

Now for any $u\in U_\tau$ we have $r_\alpha(u)=u-\alpha^\vee(u)\alpha \in U_\tau$.
We set $R_\alpha$ to be the restriction of the reflection $r_\alpha$ to $U_\tau$. By definition, $R_\alpha$
is a $k$-linear reflection on $U_\tau$ with respect to~$\alpha$.

\begin{lem}\label{lem:invcomm}
The operator $R_\alpha$ commutes with $\mathcal{T}$.
\end{lem}

\begin{proof}
Since, $\mathcal{T}^2=id$ we get
\begin{align*}
c\mathcal{T}(R_\alpha(u)) &=\mathcal{T}(cu)-2B_\tau(u,\alpha)\mathcal{T}\alpha\\
&= \mathcal{T}(cu)-2B_\tau(\mathcal{T}u,\mathcal{T}\alpha)\mathcal{T}\alpha\\
&=\mathcal{T}(cu)-2B_\tau(\mathcal{T}u,\alpha)\alpha=cR_\alpha(\mathcal{T}(u)).\qedhere
\end{align*}
\end{proof}

\begin{prop} 
For all simple roots in $\Delta_{rat}\amalg \Delta^{\mathcal{T}}$ 
the corresponding reflections
$R_{\bar\alpha}$
stabilize the subset $\pi_\tau (\Phi)$ in $U_\tau$.
\end{prop}

\begin{proof}
It follows from the fact that operators $R$ commute with $\mathcal{T}$ and, hence, with $\pi_\tau=\tfrac{1}{\tau-\sigma}(\mathcal{T}-\sigma I_l)$.
\end{proof}

 \begin{rem} Following Remark~\ref{rem:splitproj}, in the split case (e.g. if there is a $\tau$-invariant root),
we replace $\pi_\tau$ by the composite $p\circ\pi_\tau\colon U\to U_1$ and, hence, we consider
all operators $R$ as $k$-linear reflections on the $\mathcal{T}$-invariant submodule~$U_1$.
\end{rem}

\section{Twisted foldings of root systems}

Suppose there is a folded representation $U$ of the root system $\Phi$ over $k$ together with the projection $\pi_\tau\colon U\to U_\tau$.
Consider the subset \[\Delta_\tau=\pi_\tau(\Delta_{rat}\amalg \Delta^{\mathcal{T}})\subset U_\tau.\]
Let $W_\tau$ be the group generated by the corresponding reflections and
let $\Phi_\tau=W_\tau(\Delta_\tau)\subset \pi_\tau(\Phi)$. 
Since $W_\tau$ stabilizes a finite subset $\Phi_\tau$ of vectors in $l$ and is generated by the corresponding reflections
(here we use the fact that $l\subset \R$, hence, $B_\tau$ is a positive definite form over $l$) it is a finite root system in $U_\tau$ over $l$. 

In the split case, we replace everywhere $\pi_\tau$ by $p\circ\pi_\tau$, $l$ by $k$, $U_\tau$ by $U_1$, $W_\tau$ by $W_1$, $\Phi_\tau$ by $\Phi_1$ and $\Delta_\tau$ by $\Delta_1$.
Observe that (in the non-split case) $\Phi_\tau$ is not necessarily crystallographic (see Example~\ref{ex:EH}). 

\begin{dfn}\label{dfn:taufolding}
We call the root system $\Phi_\tau$ a $\tau$-twisted folding of $\Phi$.
\end{dfn}

Since $\pi_\tau$ is an isomorphism and operators $R$ commute with $\pi_\tau$, 
$W_\tau$ can be identified with a subgroup of $W$ generated
by $r_{\mathcal{T}\alpha} r_\alpha$ and $r_\beta$, where $\alpha\in \Delta_{rat}$ and $\beta\in \Delta^{\mathcal{T}}$ respectively. Moreover, we have
\[
W_\tau\subseteq W_{\mathcal{T}}=\{w\in W \mid w\mathcal{T}=\mathcal{T}w\}.
\]

We claim that
\begin{lem}
The subset $\Delta_\tau$ is a set of simple roots of the root system $\Phi_\tau$.
\end{lem}

\begin{proof}
It is enough to show that any root $\pi_\tau(\gamma)$ of $\Phi_\tau$ ($\gamma\in \Phi$) can be expressed as a positive or a negative
combination of roots from $\Delta_\tau$. Without loss of generality, we may assume $\gamma$ is a positive combination of simple roots from $\Delta$, i.e.
$\gamma=\sum_{\alpha\in \Delta} m_\alpha \alpha$, $m_\alpha\ge 0$. Then
\begin{align*}
\pi_\tau(\gamma) &=\sum_{\alpha} m_\alpha \pi_\tau(\alpha) =\sum_{\alpha\in \Delta_{rat}} (m_\alpha \pi_\tau(\alpha)+m_{\mathcal{T}\alpha} \pi_\tau(\mathcal{T}\alpha))+\sum_{\alpha\in \Delta^{\mathcal{T}}} m_\alpha \pi_{\tau}(\alpha) \\
 &=\sum_{\alpha\in \Delta_{rat}} (m_\alpha+\tau m_{\mathcal{T}\alpha})\bar \alpha +
 \sum_{\alpha \in \Delta^{\mathcal{T}\alpha}} m_\alpha \alpha.
\end{align*}
Since $m_\alpha+\tau m_{\mathcal{T}\alpha}\ge 0$ ($\tau>0$ according to Remark~\ref{rem:foldreps}), we are finished.
\end{proof}

We now provide examples of $\tau$-twisted foldings of root systems. 
In Examples~\ref{ex:AC} and \ref{ex:BD}) we deal with split cases and show that the $\tau$-twisted foldings coincide
with the classical foldings of root systems associated
to an involution of the Dynkin diagram. In Example~\ref{ex:EH} we consider a non-split twisted folding which can be identified with the projection of the root system
of type $E_8$ into the so called icosians introduced and studied in \cite{MP}.

\begin{ex}\label{ex:AC}
Let $k=\Z[\tfrac{1}{2}]$.
Consider a free $k$-module $V$ of rank $2n$, $n\ge 1$ 
with a basis $\{\epsilon_1,\ldots,\epsilon_{2n}\}$ and the respective
standard dot product $B(-,-)$ on $V$.
Consider
a symplectic involution $\chi\colon \epsilon_i\mapsto -\epsilon_{2n-i+1}$, $1\le i\le 2n$.
This involution preserves the set of (positive) roots $\Phi^+:=\{\epsilon_i-\epsilon_j\mid 1\leq i<j\leq 2n\}$ of a root system $\Phi$ of type $A_{2n-1}$ in $V$, since 
\[
\chi(\epsilon_i-\epsilon_j)=\epsilon_{2n-j+1}-\epsilon_{2n-i+1}
\]
and $2n-j+1<2n-j+1$ if $i<j$. 

Now let $\mathcal{U}$ be a free module of rank $n$ over $l=k\times k$ so that $U=R_{l/k}\mathcal{U}$ is a vector space of rank $2n$ over $k$.
Consider the $k$-linear isomorphism
$V\stackrel{\simeq}\to U$ 
given by
 \[
 (y_1,\ldots,y_{2n})\mapsto ((-y_{2n},y_1),(-y_{2n-1},y_2),\ldots,(-y_{n+1},y_n)),\quad y_i\in k.
 \]
It is an isometry over $k$ under which
 the standard inner product $(\cdot,\cdot)$ on $V$ corresponds to the $\tau$-form $B_\tau(\cdot,\cdot)$ and the operator $\mathcal{T}$ of multiplication by $\tau$ extends to the involution $\chi$ of $V$.
Moreover, under this isometry the simple roots $\alpha_i=\epsilon_i-\epsilon_{i+1}$ of $\Phi$
 correspond to
 {\small
 \begin{align*}
 \alpha_1 &\mapsto a_1 =(-\sigma,-\tau,0,0,\ldots)=((0,1),(0,-1),(0,0),\ldots), \\
 \alpha_2 &\mapsto a_2 =(0,-\sigma,-\tau,0,\ldots)=((0,0),(0,1),(0,-1),\ldots), \\
 &\ldots \\
 \alpha_{n-1}&\mapsto a_{n-1} =(0,\ldots,0,-\sigma,-\tau)=(\ldots,(0,0),(0,1),(0,-1)),\\
 \alpha_n &\mapsto a_n =(0,\ldots,0,1+\tau) =((0,0),\ldots,(0,0),(1,1)), \\
 \alpha_{n+1}&\mapsto \tau a_{n-1}=(0,\ldots,0,1,-1)=(\ldots,(0,0),(1,0),(-1,0)),\\
 &\ldots  \\
 \alpha_{2n-1}&\mapsto \tau a_1=(1,-1,0,\ldots,0)=((1,0),(-1,0),(0,0),\ldots). 
 \end{align*}}
 (Observe that $\tau a_n=a_n$).
 Hence, the set of simple roots $\Delta$ in $U$ partitions as \[\Delta_{rat}\amalg \mathcal{T}\Delta_{rat} \amalg \Delta^{\mathcal{T}}, \text{ where }
 \Delta_{rat}=\{a_1,\ldots,a_{n-1}\}\text{ and }\Delta^{\mathcal{T}}=\{a_n\}.\]
So $U$ is a folded representation of the root system $\Phi$.
 
The projection of the set $\Delta_{rat}\amalg \Delta^{\mathcal{T}}$ under $p\circ \pi_\tau$ of Remark~\ref{rem:splitproj} then gives
the finite set of $n$ vectors $\Delta_1$ in the subspace of $\chi$-invariants $U_1$ over $k$
{\small \begin{align*}
p(\pi_\tau(a_1))&=(\tfrac{1}{2}(1,1),\tfrac{1}{2}(-1,-1),(0,0),\ldots), \\
p(\pi_\tau(a_{n-1}))&=((0,0),\ldots,\tfrac{1}{2}(1,1),\tfrac{1}{2}(-1-1)), \\
p(\pi_\tau(a_n))&=(\ldots,(0,0),(1,1)).
\end{align*}}
The latter coincides with the set of simple roots of a root system of type $C_n$ in $U_1$.
 Hence, the $\tau$-twisted folding of $\Phi$ coincides with the classical folding of the root system of type $A_{2n-1}$ to the root system of type $C_n$ (see e.g. \cite{St67}).
\end{ex}

\begin{ex}\label{ex:BD}
Similarly, let $V$ be a free module of rank $(n+1)$ over $k=\Z[\tfrac{1}{2}]$. Again
we denote by $\{\epsilon_1, \epsilon_2, \ldots, \epsilon_{n+1}\}$  a basis of $V$. 
Consider an orthogonal involution on $V$ given by
\[
\chi\colon\epsilon_i\mapsto 
\left\{
\begin{array}{ll}
\epsilon_i&\hbox{if } i\neq n+1,\\
-\epsilon_{n+1}&\hbox{if } i= n+1.
\end{array}
\right.
 \]
 We embed a root system $\Phi$ of type $D_{n+1}$ into $V$ 
 by identifying the set of simple roots with 
\[
\Delta=\{\alpha_i:=\epsilon_i-\epsilon_{i+1}\mid i=1, \ldots, n\}\cup\{\alpha_{n+1}:=\epsilon_{n}+\epsilon_{n+1}\}.
\]
Then
\[
\chi(\alpha_i)=\left\{
\begin{array}{ll}
\alpha_i&\hbox{if }i< n,\\
\alpha_{n+1}&\hbox{if }i=n, \\
\alpha_{n}&\hbox{if }i=n+1.
\end{array}
\right.
\]
So the involution $\chi$ stabilizes the set of simple roots and, therefore, the set of positive roots $\Phi^+$ of the root system $\Phi$. 

Let $\mathcal{U}$ be a free module of rank $n$ over $l=k\times k$ so that $U=R_{l/k}\mathcal{U}$ is a free module of rank $2n$ over $k$.
Consider the $k$-linear inclusion
$V\hookrightarrow U$ given by 
{\small \begin{align*}
 \epsilon_1 &\mapsto ((1,1),(0,0),\ldots),\\
\epsilon_2 &\mapsto ((0,0),(1,1),\ldots),\\
& \ldots \\
\epsilon_{n} &\mapsto (\ldots,(0,0),(1,1)),\\
\epsilon_{n+1} &\mapsto (\ldots,(0,0),(1,-1)).
\end{align*}}
The restriction to $V$ of the operator $\mathcal{T}$ 
of multiplication by $\tau$ then coincides with the  involution $\chi$.
Under the inclusion the simple roots are
{\small \begin{align*}
 \alpha_1 &\mapsto a_1 =(1+\tau,-1-\tau,0,0,\ldots)=((1,1),(-1,-1),(0,0),\ldots), \\
 \alpha_2 &\mapsto a_2 =(0,1+\tau,-1-\tau,0,\ldots)=((0,0),(1,1),(-1,-1),\ldots), \\
 &\ldots\\
 \alpha_{n-1}&\mapsto a_{n-1} =(\ldots,0,1+\tau,-1-\tau)=(\ldots,(0,0),(1,1),(-1,-1)),\\
 \alpha_n &\mapsto a_n =(0,\ldots,0,2\tau)=(\ldots,(0,0),(0,2)), \\
 \alpha_{n+1}&\mapsto \tau a_{n}=(0,\ldots,0,2)=(\ldots,(0,0),(2,0)).
 \end{align*}}
 (Observe that $\tau a_i=a_i$ for all $i<n$).
So $\Delta$ partitions as \[\Delta^{\mathcal{T}}\amalg \Delta_{rat}\amalg \mathcal{T}\Delta_{rat}, \text{ where }
 \Delta_{rat}=\{a_n\}\text{ and }\Delta^{\mathcal{T}}=\{a_1,\ldots,a_{n-1}\}.\]
 Hence, we obtain the twisted folded representation of the root system of type $D_{n+1}$.

Finally, 
the projection of the set $\Delta_{rat}\amalg \Delta^{\mathcal{T}}$ under $p\circ\pi_\tau$ of Remark~\ref{rem:splitproj} gives
the finite set $\Delta_1$ of $n$ vectors in $U_1$ over $k$
{\small \begin{align*}
p(\pi_\tau(a_1))&=((1,1),(-1,-1),(0,0),\ldots),\\
&\ldots \\
p(\pi_\tau(a_{n-1}))&=(\ldots,(0,0),(1,1),(-1,-1)),\\
p(\pi_\tau(a_n))&=(\ldots,(0,0),(1,1)).
 \end{align*}}
which coincides with the set of simple roots of type $B_n$ in $U_1$.
Therefore, our folded representation gives (up to a rescaling) a classical folding of the root system of type $D_{n+1}$ to the root system of type $B_n$ (see e.g. \cite{St67}).
\end{ex}

\begin{ex} \label{ex:EH}
Let $V$ be a free module of rank 8 over $k=\Z[\tfrac{1}{5}]$.
Consider a root system $\Phi$ of type $E_8$.
\[
\xymatrix{
\bullet^{\alpha_1}\ar@{-}[r] & \bullet^{\alpha_2} \ar@{-}[r] & \bullet^{\alpha_3} \ar@{-}[r]& \bullet^{\alpha_4} \ar@{-}[dl]\\
\bullet_{\alpha_7} \ar@{-}[r]& \bullet_{\alpha_6} \ar@{-}[r]& \bullet_{\alpha_5} \ar@{-}[r]& \bullet_{\alpha_8}
}
\]

 It can be realized 
as a subset of vectors in $V$ with simple roots given by
 {\small \begin{align*}
 \alpha_1&=(2,-2,0,0,0,0,0,0)\qquad &  \alpha_5&=(0,0,0,0,2,-2,0,0) \\
 \alpha_2&=(0,2,-2,0,0,0,0,0)    &\alpha_6&=(0,0,0,0,0,2,2,0) \\
 \alpha_3&=(0,0,2,-2,0,0,0,0)   &\alpha_7&=-(1,1,1,1,1,1,1,1)\\
 \alpha_4&=(0,0,0,2,-2,0,0,0)     &\alpha_8&=(0,0,0,0,0,2,-2,0)
 \end{align*}}
Let $\mathcal{U}$ be a free module of rank $4$ over $l=\Z(\tau)[\tfrac{1}{5}]$, where
$\Z(\tau)$ is the ring appearing in the Example~\ref{ex:golden} so that 
$U=R_{l/k}\mathcal{U}$ is a free module of rank 8 over $k$.
Consider the realization of the root system $\Phi$ inside $U$ 
of \cite[(3.3)]{MP} given by the so called icosians and their $\tau$-multiples
{\small \begin{align*}
 \alpha_1&\mapsto a_1=(-\sigma,-\tau,0,-1)=((-1,1),(0,-1),(0,0),(-1,0)) \\
 \alpha_2&\mapsto a_2=(0,-\sigma,-\tau,1)=((0,0),(-1,1),(0,-1),(1,0)) \\
 \alpha_3&\mapsto a_3=(0,1,-\sigma,-\tau)=((0,0),(1,0),(-1,1),(0,-1)) \\
 \alpha_4&\mapsto \tau a_4=(0,-\tau,1,\tau^2)=((0,0),(0,-1),(1,0),(1,1)) \\
 \alpha_5&\mapsto \tau a_3=(0,\tau,1,-\tau^2)=((0,0),(0,1),(1,0),(-1,-1)) \\
 \alpha_6&\mapsto \tau a_2=(0,1,-\tau^2,\tau)=((0,0),(1,0),(-1,-1),(0,1)) \\
 \alpha_7&\mapsto \tau a_1=(1,-\tau^2,0,-\tau)=((1,0),(-1,-1),(0,0),(0,-1))\\
 \alpha_8&\mapsto a_4=(0,-1,-\sigma,\tau)=((0,0),(-1,0),(-1,1),(0,1)).
 \end{align*}}
The operator $\mathcal{T}$ is an automorphism of the root lattice of $\Phi$
given by the formulas \cite[(4.2)]{MP} (note that it does not preserve $\Phi$ itself). 
Moreover, the set of simple roots partitions as
$\Delta_{rat}\amalg \mathcal{T}\Delta_{rat}$, where $\Delta_{rat}=\{a_1,a_2,a_3,a_4\}$ (here $\Delta^{\mathcal{T}}=\emptyset$).
The respective $\tau$-twisted folding $\Phi_\tau$
of $\Phi$ is a non-crystallographic root system over $l$
where the corresponding reflection group $W_\tau$ is a group
of type $H_4$ (see \cite[\S6]{MP}).
\end{ex}

\section{Twisted foldings and moment graphs}

Consider a root datum $\Phi\hookrightarrow \Lambda^\vee$ with a set of simple roots $\Delta$.
Let $(U,\mathcal{T})$ be a folded representation. Let $\Theta$ be a subset of $\Delta$.
\begin{dfn}  We say that $(U,\mathcal{T})$ preserves $\Theta$ if 
the operator $\mathcal{T}$ preserves  the $k$-submodule of $U$ generated by all roots from $\Theta$, i.e., 
\[
\mathcal{T}(\langle \alpha\mid \alpha\in \Theta\rangle_k)=
\langle \alpha\mid \alpha\in \Theta\rangle_k.
\]
\end{dfn}
Observe that the partition of $\Delta$ induces a partition of $\Theta$
\[
\Theta=\Theta^{\mathcal{T}}\amalg\Theta_{rat}\amalg \mathcal{T}(\Theta_{rat}),
\text{ where }\Theta_{rat}=\Delta_{rat}\cap \Theta.
\]
In the remaining of this section we assume that $(U,\mathcal{T})$ preserves $\Theta$.

Denote by $U^+_\Delta$ the dominant Weyl chamber for $\Delta\subset U$ over $k$.
Let $W_\Theta$ be the parabolic subgroup of $W$ generated by all reflections corresponding to roots from~$\Theta$.
We choose a vector $\theta\in U^+_\Delta$ over $k$ such that the stabilizer subgroup of $W$ at $\theta$ is $W_\Theta$, i.e.,
 $w\theta=\theta$ for $w\in W$ if and only if $w\in W_\Theta$. 
Since $\theta$ is dominant, we have $B_\tau(\theta,\beta)\ge 0$ for all $\beta\in \Delta$. Then for all $\alpha \in \Delta_{rat}$ we obtain
\begin{align*}
B_\tau(\pi_\tau(\theta),\pi_\tau(\alpha)) &=\tfrac{1}{(\tau-\sigma)^2}B_\tau((\mathcal{T}-\sigma I_l)(\theta),(\mathcal{T}-\sigma I_l)(\alpha)) \\
&=\tfrac{1}{(\tau-\sigma)^2}\big(B_\tau(\mathcal{T}\theta,\mathcal{T}\alpha)-2\sigma B_\tau(\theta,\mathcal{T}\alpha)+\sigma^2 B_\tau(\theta,\alpha)  \big)\\
&=\tfrac{1}{(\tau-\sigma)^2}\big( B_\tau(\theta,(c_1\mathcal{T}+1)\alpha)-2\sigma B_\tau(\theta,\mathcal{T}\alpha)+\sigma^2 B_\tau(\theta,\alpha)    \big)\\
&=\tfrac{1}{(\tau-\sigma)^2}\big((1+\sigma^2)B_\tau(\theta,\alpha)+(\tau-\sigma)B_\tau(\theta,\mathcal{T}\alpha)  \big)\ge 0.
\end{align*}

For all $\alpha\in \Delta^{\mathcal{T}}$ (hence, in the split case) we also get
\begin{align*}
B_\tau(\pi_\tau(\theta),\pi_\tau(\alpha))&=\tfrac{\tau}{2}B_\tau((\mathcal{T}-\sigma I_l)(\theta),\alpha)\\
&=\tfrac{\tau}{2}(1-\sigma)B_\tau(\theta,\alpha)\ge 0.
\end{align*}
Notice also that $B_\tau(\theta,\alpha)=0$  if and only if $B_\tau(\pi_\tau(\theta),\pi_\tau(\alpha))=0$.
So we obtain the following

\begin{lem}\label{lem:domvect}
Let $\theta\in U$ be a dominant vector with respect to $\Delta$ over $k$ such that the subgroup of $W$ stabilizing $\theta$ is $W_\Theta$.

Then its image $\mu=\pi_\tau (\theta) \in U_\tau$ is a dominant vector with respect to $\Delta_\tau$ over $l$ 
such that the subgroup of $W_\tau$ stabilizing $\mu$ is $(W_\Theta)_\tau$.
\end{lem}

We recall the definition of a moment graph following \cite[Definition~2.1]{DLZ}
 
\begin{dfn} The data $\mathcal{G}=((\mathcal{V},\le), l\colon E\to \Phi^+)$ is called a moment graph if
\begin{itemize}
\item[(MG1)] $\mathcal{V}$ is a set of vertices together with a partial order `$\le$', i.e., we are given a poset $(\mathcal{V},\le)$.
\item[(MG2)] $E$ is a set of directed edges labelled by positive roots $\Phi^+$ via the label function $l$, i.e.,
an edge $x\to y\in E$ is labelled by $l(x \to y)\in \Phi^+$.
\item[(MG3)] For any edge $x\to y\in E$, we have $x\le y$, i.e., direction of edges respects the partial order.
\end{itemize}
\end{dfn}

A basic example of a moment graph can be given as follows (cf. \cite[Example~2.3]{DLZ}).
Consider a $W$-orbit $W\theta$ of $\theta$. 
Since $\theta$ is a dominant vector, the Bruhat order on cosets $W/W_\Theta$ corresponds to the partial order on $W\theta$ given by the transitive closure of the following relations (see \cite[Proposition~1.1]{St05}):
\[
w\theta<r_\alpha(w\theta)\text{ for all }\alpha\in \Phi^+\text{ such that }B_\tau(w\theta,\alpha)>0.
\]
Analogously by Lemma~\ref{lem:domvect} there is an induced partial order on the cosets $W_\tau/(W_\Theta)_\tau$ and, hence, on the orbit
$W_\tau\mu \subset U_\tau$ of $\mu=\pi_\tau(\theta)$.

We identify the orbit poset $W\theta$ (resp. the orbit poset $W_\tau\mu$) with the set of vertices $\mathcal{V}$ (resp. with the set of vertices $\mathcal{V}_\tau$) of 
the associated moment graph. This produces an injective map between posets of vertices
$\iota\colon \mathcal{V}_\tau\hookrightarrow \mathcal{V}$.

As for edges we set
\begin{align*}
E &=\{w\theta\to r_\alpha w\theta \mid w\theta \le r_\alpha w\theta,\; \alpha\in \Phi^+\}\;\text{ and }\;
l(w\theta\to r_\alpha w\theta)=\alpha \\
E_\tau &=\{w\mu\to R_\gamma w\mu \mid w\mu\le R_\gamma w\mu,\; \gamma\in \Phi_\tau^+\}\;\text{ and }\;
l(w\mu\to R_\gamma w\mu)=\gamma.
\end{align*}
Observe that
the map $\iota$ does not preserve edges and, therefore, it is not a map of moment graphs.

\section{Twisted foldings and structure algebras}

For a general moment graph $\mathcal{G}$ one can define the so called structure algebra $\mathcal{Z}$ (see \cite{F1}) by
\[
\mathcal{Z}(\mathcal{G})=
\left\{
(z_x)\in\bigoplus_{x\in \mathcal{V}} S\mid
\begin{array}{c}z_x-z_{y}\in l(x\to y) S\\
\forall x\to y \in E \end{array}
\right\}, 
\]
where $S=Sym_R(\Lambda)$ over some coefficient ring $R$. 

It is known that the structure algebra $\mathcal{Z}(\mathcal{G})$ coincides with the module of global sections of the structure sheaf of 
$\mathcal{G}$ \cite{BMP,F1}. 
On the other side, it also computes the equivariant Chow ring of flag varieties \cite{KK86}.
In this paper we will not discuss foldings in the context of moment graph sheaves but rather limit ourselves to the study of structure algebras and its relation to equivariant Chow rings.

For the moment graphs $\mathcal{G}$ and $\mathcal{G}_\tau$ of the previous section we have (see \cite{DLZ})
\[
\mathcal{Z}(\mathcal{G})=
\left\{
(z_x)\in\bigoplus_{x\in W\theta} S\mid
\begin{array}{c}z_x-z_{r_\alpha x}\in \alpha S\\
\forall x\in W\theta,\; \forall \alpha\in\Phi^+\end{array}
\right\},
\]
where $S=Sym_k(\Lambda)$, and
\[
\mathcal{Z}(\mathcal{G}_\tau)=
\left\{
(b_y)\in\bigoplus_{y\in W_\tau\mu} S_\tau\mid
\begin{array}{c}b_y-b_{R_\beta y}\in \beta S_\tau\\
\forall y\in W_\tau\mu,\; \forall \beta\in\Phi_\tau^+\end{array}
\right\},
\]
where 
\[S_\tau=Sym_l(\Lambda_\tau)=l\oplus \Lambda_\tau \oplus \Lambda_\tau^2 \oplus \ldots,\quad \Lambda_\tau:=\pi_\tau(\Lambda) \subset U_\tau.\] 
(since $\pi_\tau(\mathcal{T}\alpha)=\tau\pi_\tau(\alpha)$, $\Lambda_\tau$ is a free $l$-linear submodule of $U_\tau$). 
The map $\pi_\tau$ induces a homomorphism of graded $k$-algebras $S\to S_\tau$. By abuse of notation we also denote it by $\pi_\tau$. 

Consider the augmentation map $\epsilon\colon S\to k$, $\alpha \mapsto 0$ for all $\alpha\in \Lambda$. Let $\II$ denote its kernel.
So $\II$ is the (augmentation) ideal in $S$ consisting of polynomials with trivial constant terms. 
The structure algebra $\mathcal{Z}(\mathcal{G})$ has a natural structure of $S$-module via \[s\cdot (z_x):=(sz_x),\; s\in S.\]
Its quotient modulo the ideal $\II \cdot \mathcal{Z}(\mathcal{G})$ is denoted by $\overline{\mathcal{Z}}(\mathcal{G})$ and called the augmented structure algebra. 
The quotient map $\mathcal{Z}(\mathcal{G}) \to \overline{\mathcal{Z}}(\mathcal{G})$ is denoted by $f$.
Similarly we set $\II_\tau$ to be the (augmentation) ideal  for $\epsilon_\tau\colon S_\tau\to l$, $\overline{\mathcal{Z}}(\mathcal{G}_\tau)$ to be the quotient of $\mathcal{Z}(\mathcal{G}_\tau)$ modulo $\II_\tau\cdot  \mathcal{Z}(\mathcal{G}_\tau)$ and $f_\tau\colon \mathcal{Z}(\mathcal{G}_\tau) \to \overline{\mathcal{Z}}(\mathcal{G}_\tau)$ to be the quotient map.

\begin{rem} Given an adjoint split semisimple linear algebraic group $G$, a split maximal torus $T$ and a parabolic subgroup $P$ of $G$, consider the respective root system $\Phi$ together with the subset $\Theta\subset \Delta$, where $P=P_\Theta$ corresponds to $\Theta$. 
Then by the main result of \cite{KK86} the structure algebra of the associated moment graph $\mathcal{G}$ computes the $T$-equivariant Chow ring of the projective homogeneous variety $G/P$, i.e. \[\mathcal{Z}(\mathcal{G})\simeq CH_T(G/P).\]
The augmented algebra $\overline{\mathcal{Z}}(\mathcal{G})$  coincides with the usual Chow ring $CH(G/P)$  and the quotient map
$f\colon \mathcal{Z}(\mathcal{G})\to \overline{\mathcal{Z}}(\mathcal{G})$ is the forgetful map $CH_T(G/P) \to CH(G/P)$ (see e.g. \cite[(5.12)]{KK86}).
\end{rem}

 We are now ready to prove the following
 
 \begin{thm}\label{thm:mainthm}
 Given a folded representation $(U,\mathcal{T})$ that preserves $\Theta\subset\Delta$,
there is an induced ring homomorphism $\iota^*\colon \mathcal{Z}(\mathcal{G})\to \mathcal{Z}(\mathcal{G}_\tau)$ between the structure algebras of the associated moment graphs given by the composite 
 \[
\begin{array}{ccccc}
\mathcal{Z}(\mathcal{G})&\rightarrow&\bigoplus_{w\in W_\tau\mu}S&\rightarrow& \bigoplus_{w\in W_\tau\mu}S_\tau .\\
(z_x)_{x\in W\theta}&\mapsto&(z_y)_{y\in W_\tau\mu}&\mapsto&(\pi_\tau (z_y))_{y\in W_\tau\mu}
\end{array}
\] 
Moreover, it restricts to the augmented structure algebras so that there is a commutative diagram of 
$k$-algebra homomorphisms
 \[
 \xymatrix{ 
  \mathcal{Z}(\mathcal{G}) \ar[d]_{\iota^*}\ar[r]^f & \overline{\mathcal{Z}}(\mathcal{G}) \ar[d]^{\bar \iota^*}\\
  \mathcal{Z}(\mathcal{G}_\tau) \ar[r]^{f_\tau}& \overline{\mathcal{Z}}(\mathcal{G}_\tau).
 }
 \]
\end{thm}
 
\begin{proof}
 Let $(z_x)_{x\in W\theta}\in \mathcal{Z}(\mathcal{G})$ and for $y\in W_\tau\mu$ denote $b_y:=\pi_\tau(z_y)\in S_\tau$. We have hence to show that for all $y\in W_\tau\mu$ and for all $\beta\in\Phi_\tau^+$ we have $b_y-b_{R_\beta y}\in\beta S_\tau$.
 
 There exist $w\in W_\tau$ and $\alpha \in \Delta_{rat}\amalg \Delta^{\mathcal{T}}$ such that $\beta=w(\bar\alpha)$ and $R_\beta=wR_{\bar\alpha}w^{-1}$. 
 
 If $\alpha\in \Delta^{\mathcal{T}}$, then $\beta=w(\alpha)$ as $\bar\alpha=\alpha\in U_\tau$ and $R_\beta=w r_\alpha w^{-1}=r_{w(\alpha)}$ which is also a reflection in $W$, so that
 \[
 b_y-b_{R_\beta y}=\pi_\tau(z_y)-\pi_\tau(z_{r_{w(\alpha)}y})=\pi_\tau(z_y-z_{r_{w(\alpha)}y})\in \pi_\tau(w(\alpha))S_\tau.
 \]
 As $\pi_\tau$ commutes with $w\in W_\tau$ (see Lemma~\ref{lem:commuting} and~\ref{lem:invcomm}), the latter coincides with \[\pi_\tau(w(\alpha))S_\tau=w(\bar\alpha)S_\tau=w(\alpha)S_\tau=\beta S_\tau.\]
 
 If $\alpha\in \Delta_{rat}$ we get
 \[
 R_\beta=wR_{\bar\alpha}w^{-1}=wr_{\mathcal{T}\alpha}r_\alpha w^{-1}=w r_{\mathcal{T}\alpha} w^{-1}w r_{\alpha}w^{-1}=r_{w(\mathcal{T}\alpha)}r_{w(\alpha)}.
 \]
We have for all $x\in W\theta$
\[
z_x-z_{r_{w(\alpha)}x}\in w(\alpha)S, \quad  z_{r_{w(\alpha)}x}-z_{r_{w(\mathcal{T}\alpha)} r_{w(\alpha)}x}\in w(\mathcal{T}\alpha)S,
\]

 As $\pi_\tau$ commutes with $w\in W_\tau$  we obtain
\[
b_y-b_{R_\beta y}=\pi_\tau(z_y-z_{r_{w(\mathcal{T}\alpha)}r_{w(\alpha)}y})\in \big(w(\pi_\tau(\alpha)), w(\pi_\tau(\mathcal{T}\alpha))\big)S_\tau.
\]
But $\pi_\tau(\mathcal{T}\alpha)=\tau \pi_\tau(\alpha)=\tau\bar\alpha$ and $\tau\in S_\tau$, so
\[
b_y-b_{R_\beta y} \in (w(\bar\alpha),\tau w(\bar\alpha))S_\tau=w(\bar\alpha)S_\tau=\beta S_\tau. 
\]

Finally, since $\pi_\tau$ maps $\II$ onto $\II_\tau$ (it is surjective in positive degrees), the map $\iota^*$ restricts to the quotients.
\end{proof}

 \begin{ex}\label{ex:mapH}
 Consider the twisted folding $(U,\mathcal{T})$ over $k=\Z[\tfrac{1}{5}]$ of Example~\ref{ex:EH}. 
 By formulas \cite[(4.2)]{MP} it preserves 
 any $\Theta$ which is a union of some of the pairs $\{\alpha_1,\alpha_7\}$,
 $\{\alpha_2,\alpha_6\}$, $\{\alpha_3,\alpha_5\}$ and $\{\alpha_4,\alpha_8\}$.
 Hence, applying the theorem
 we obtain a ring homomorphism 
 \[\iota^*\colon\mathcal{Z}(\mathcal{G}) \to \mathcal{Z}(\mathcal{G}_\tau).\] 
The left hand side can be identified with the
 $T$-equivariant Chow group (resp. cohomology with complex coefficients)  of the projective homogeneous variety for $G$ of type $E_8$ which we denote by $E_8/P_\Theta$. So we have
 \[\mathcal{Z}(\mathcal{G})=CH_T(E_8/P_\Theta; k),\; k=\Z[\tfrac{1}{5}].\]
 As for the right hand side, $\mathcal{Z}(\mathcal{G}_\tau)$ can be viewed as a virtual equivariant Chow ring of a virtual flag $H_4/P_{\Theta_\tau}$ with coefficients in $l=\Z(\tau)[\tfrac{1}{5}]$, $\tau=\tfrac{1}{2}(1+\sqrt{5})$
 (note that neither the algebraic group of type $H_4$, nor 
the variety $H_4/P_{\Theta_\tau}$, nor its Chow group exist). To stress this similarity we denote it by $\mathcal{Z}(H_4/P_{\Theta_\tau};l)$

In particular, if we take $\Theta=\{\alpha_2,\alpha_3,\alpha_4,\alpha_5,\alpha_6,\alpha_8\}$, then $P_\Theta$ is a parabolic subgroup of type $D_6$ 
and the respective $\Theta_\tau$ corresponds to a subsystem of type $H_3$. So we obtain a commutative diagram of ring homomorphisms
\[
 \xymatrix{ 
 CH_T(E_8/P_{D_6}; k)        \ar[d]_{\iota^*}\ar[r]^f &       CH(E_8/P_{D_6}; k)         \ar[d]^{\bar\iota^*}\\
\mathcal{Z}(H_4/P_{H_3}; l)   \ar[r]^{f_\tau}& \overline{\mathcal{Z}}(H_4/P_{H_3}; l).
 }
 \]
  \end{ex}

We now study behaviour of the map $\iota^*$ with respect to equivariant characteristic maps $c$, $c_\tau$ and Borel maps $\rho$, $\rho_\tau$ defined as follows 
(see e.g. \cite{DLZ})
\begin{align*}
c &\colon S^{W_\Theta} \to \mathcal{Z}(\mathcal{G}),\;  s\mapsto (w(s))_{w\theta},\\
c_\tau &\colon S_\tau^{W_{\Theta_\tau}} \to \mathcal{Z}(\mathcal{G}_\tau),\;  s\mapsto (w(s))_{w\mu},\;  \text{ and}\\
\rho &\colon S\otimes_{S^W} S^{W_\Theta} \to \mathcal{Z}(\mathcal{G}), \; s_1\otimes s_2\mapsto s_1c(s_2),\\
\rho_\tau &\colon S_\tau\otimes_{S_\tau^{W_\tau}} S_\tau^{W_{\Theta_\tau}} \to \mathcal{Z}(\mathcal{G}_\tau), \;  s_1\otimes s_2\mapsto s_1c_\tau(s_2).
\end{align*}
Observe that by definition there is a commutative diagram
\[
\xymatrix{ 
 S\otimes_{S^W} S^{W_\Theta} \ar[r]^-{\rho} \ar[d]_{\epsilon \otimes\mathrm{id}} &  \mathcal{Z}(\mathcal{G})\ar[d]^{f} \\
k \otimes_{S^{W}} S^{W_{\Theta}} \ar[r]^-{\bar\rho} &  \overline{\mathcal{Z}}(\mathcal{G}).
 },
 \]
 where $\bar \rho$ is (the usual characteristic map) defined by $r\otimes s \mapsto rc(s)$, $r\in k$, $s\in S^{W_{\Theta}}$. 
 Moreover, there is a similar diagram involving $S_\tau$ and $\overline{\mathcal{Z}}(\mathcal{G}_\tau)$.

 Since reflections of $W$ and $W_\tau$ commute with $\pi_\tau$, combining the above diagrams we obtain

\begin{lem}\label{lem:charm} There are commutative diagrams of $k$-algebra homomorphisms
 {\small \[
 \xymatrix{ 
 & k \otimes_{S^{W}} S^{W_{\Theta}} \ar[rr]^{\bar\rho} \ar[dd]^(0.3){\pi_\tau}& &  \overline{\mathcal{Z}}(\mathcal{G}) \ar[dd]^{\bar\iota^*} \\
 S\otimes_{S^W} S^{W_\Theta} \ar[rr]^(0.4){\rho} \ar[ru]^{\epsilon\otimes \mathrm{id}} \ar[dd]_{\pi_\tau\otimes\pi_{\tau}} & &  \mathcal{Z}(\mathcal{G})\ar[dd]^(0.6){\iota^*} \ar[ru]^{f}\\
& l \otimes_{S_\tau^{W_\tau}} S_\tau^{W_{\Theta_\tau}} \ar[rr]^(0.4){\bar\rho_\tau} & &  \overline{\mathcal{Z}}(\mathcal{G}_\tau) \\
 S_\tau \otimes_{S_\tau^{W_\tau}} S_\tau^{W_{\Theta_\tau}} \ar[ru]^{\epsilon_\tau\otimes \mathrm{id}} \ar[rr]^-{\rho_\tau} & &  \mathcal{Z}(\mathcal{G}_\tau) \ar[ru]_{f_\tau}.
 }
 \]}
\end{lem}

\begin{cor}\label{cor:surjc}
Suppose $|W_\Theta|$ is invertible in $k$ and the Borel map $\rho_\tau$ is surjective. Then the map
$\iota^*$ (resp. $\bar \iota^*$) restricted to the image of $\rho$ (resp. $\bar\rho$) is surjective in positive degrees.
\end{cor}
We call the image of $\rho$ (and of $\bar \rho$) the characteristic subring of the respective (augmented) structure algebra.

\begin{proof}
Since $|W_\Theta|$ is invertible, $\pi_\tau$ maps invariants $\II^{W_\Theta}$ onto $\II_\tau^{W_{\Theta_\tau}}$. So $\pi_\tau\otimes\pi_\tau$ is surjective in positive degrees. 
The result then follows by the lemma.
\end{proof}

\begin{rem} \label{rem:surj}
If $k=\mathbb{Q}$ then according to \cite{Hi82} and \cite{Ka11} both Borel maps $\rho$ and $\rho_\tau$ are isomorphisms. So in this case $\iota^*$ and $\bar \iota^*$ are surjective in positive degrees.

For a general $k$ and $\Theta=\emptyset$, $\rho$ is surjective if and only if the torsion index of \cite[\S5]{De73} 
of the respective linear algebraic group $G$ is invertible in $k$ (see \cite[Thm~10.2]{CZZ}). 
\end{rem}

Now consider the twisted action of $S^{W_\Theta}$ on the structure algebra $\mathcal{Z}(\mathcal{G})$ defined by (to distinguish it from the canonical action we write it on the right)
\[
(z_x)\cdot s=(z_x x(s)), \; s\in S^{W_\Theta}.
\]
By definition $\rho(s_1\otimes s_2)=s_1\cdot (1) \cdot s_2$, where $(1)$ is the unit in $\mathcal{Z}(\mathcal{G})$.
We denote the quotient of $\mathcal{Z}(\mathcal{G})$ modulo the ideal $\II\cdot \mathcal{Z}(\mathcal{G})+\mathcal{Z}(\mathcal{G}) \cdot \II^{W_\Theta}$ by $\widehat{\mathcal{Z}}(\mathcal{G})$
and call it the reduced structure algebra of $\mathcal{G}$. 
Observe that $\widehat{\mathcal{Z}}(\mathcal{G})$ is also the quotient of $\overline{\mathcal{Z}}(\mathcal{G})$ modulo the ideal $\overline{\mathcal{Z}}(\mathcal{G})\cdot \II^{W_\Theta}$ that is the ideal
generated by the image $\bar\rho(\II^{W_\Theta})$.
By Lemma~\ref{lem:charm} we then obtain 
\begin{cor}
The map $\iota^*$ induces the map on the quotients $\hat\iota^*\colon \widehat{\mathcal{Z}}(\mathcal{G}) \to \widehat{\mathcal{Z}}(\mathcal{G}_\tau)$
\end{cor}

\begin{rem}
In the case of a crystallographic root system and $\Theta=\emptyset$
the reduced structure algebra $\widehat{\mathcal{Z}}(\mathcal{G})$ can be identified with the Chow ring $CH(G;k)$ of the associated adjoint linear algebraic group
$G$ (see e.g. \cite[\S4]{Gr58} and \cite[\S3, Example~2]{Br98} when $k=\Z$).
Then the reduction map is simply a pull-back (here $G$ acts diagonally on the product)
\[
CH_G(G/B\times G/B;k)=CH_T(G/B;k) \to CH(G;k)=CH_G(G\times G;k).
\]
\end{rem}

\begin{ex} Consider the classical folding $A_{2n-1} \mapsto C_n$, $n\ge 2$ with $k=\Z[\tfrac{1}{2}]$.
Suppose $\Theta=\emptyset$ (the Borel case).

Since $2$ is the only torsion prime of the adjoint form of the associated group $G$ of type $C_n$, the maps $\iota^*$, $\bar\iota^*$
restricted to characteristic subrings
are surjective as well as the map $\hat\iota^*$.

More precisely, for the type $A_{2n-1}$ the algebra $\overline{\mathcal{Z}}(\mathcal{G})$ can be identified with the Chow ring $CH(PGL_{2n}/B;k)$ of the variety of complete flags which is generated by the classes of divisors $\{D_i\}_i$ of line bundles corresponding to fundamental weights $\{\omega_i\}_i$. On the other side it is also isomorphic to the coinvariant algebra 
$k[\omega_1,\ldots,\omega_{2n}]_W$ which is the quotient of the polynomial ring
in $\omega_i$s modulo the ideal generated by $W$-invariant polynomials with trivial constant terms.

Since $\Lambda$ is the root lattice, the characteristic subring (the image of $\bar\rho$) is
the subring of $CH(PGL_{2n}/B;k)$ generated by linear combinations $\{\sum_j c_{ij} D_j\}_i$, where $c_{ij}$ are coefficients of the Cartan matrix (i.e. $\alpha_i=\sum_{j}c_{ij}\omega_j$ for a simple root $\alpha_i$). On the other side it is also isomorphic to the coinvariant algebra $k[\alpha_1,\ldots,\alpha_{2n}]_W$. 

The map $\bar\iota^*$ is a surjective $k$-algebra homomorphism onto the Chow ring of the variety of complete symplectic flags:
\[\bar\iota^*\colon k[\alpha_1,\ldots,\alpha_{2n}]_W \longrightarrow CH(PSp_{2n}/B;k),\; k=\Z[\tfrac{1}{2}].\]
The map $\hat\iota^*$ is just the trivial map
\[
\hat\iota^*\colon CH(PGL_{2n};k) \to CH(PSp_{2n};k)=k.
\]
\end{ex}

\begin{ex}
Consider the classical folding $E_6 \to F_4$ with $k=\Z[\tfrac{1}{2}]$.
In this case we obtain a map $\hat\iota^*\colon CH(PE_6;k) \to CH(F_4;k)$ from the adjoint group of type $E_6$.
According to \cite[Cor.~6.2 and (6.12)]{DZ15} we have
\[
CH(PE_6;k)\simeq k[\omega_1,x_4]/(3\omega_1,3x_4,\omega_1^9,x_4^3),
\]
where $\omega_1\in CH^1(PE_6/B)$ and $x_4\in CH^4(PE_6/B)$; and 
\[
CH(F_4;k)\simeq k[y_4]/(3y_4,y_4^3), 
\]
where $y_4\in CH^4(F_4/B)$. Moreover, the description of the classes $x_4$ and $y_4$ of \cite[(3.3)]{DZ15} implies that 
 $\hat\iota^*$ maps the class $x_4$ from the image $CH^4(PE_6/P_{A_6}) \hookrightarrow CH^4(E_6/B)$ 
to the class $y_4$ from the image $CH^4(F_4/P_{C_3}) \hookrightarrow CH^4(F_4/B)$. 
Hence, the map $\hat \iota^*$ is surjective.
\end{ex}

Consider again the Lusztig folding of Example~\ref{ex:mapH}. It restricts to the map from the Chow ring of the group $G$ of type $E_8$:
\[
\hat\iota^*\colon CH(E_8;k) \longrightarrow \widehat{Z}(H_4;l),\quad k=\Z[\tfrac{1}{5}],\; l=k(\tau).
\]
By \cite[Cor.~6.2]{DZ15} we have
\begin{equation}\label{eq:forme8}
CH(E_8;k)\simeq \frac{k[x_3,x_4,x_5,x_9,x_{10},x_{15}]}
{(2x_3,3x_4,2x_5,2x_9,3x_{10},2x_{15},x_3^8,x_4^3,x_5^4,x_9^2,x_{15}^2-x_{10}^3)},
\end{equation}

Assume that $\hat\iota^*$ is surjective and
the reduced structure algebra $\widehat{Z}(H_4;l)$ 
has a mod $p$ presentation similar to that of  \cite[Thm.~3.(ii)]{Ka85}, i.e.
\begin{itemize}
\item[(i)]
For each torsion prime $p=2,3$ a $p$-generator $y_{d_{i,p}}$ of degree $d_{i,p}$ 
satisfies the relation $y_{d_{i,p}}^{p^{k_i}}=0$, where $(p,d_{i,p})=1$ and 
$d_i=d_{i,p}p^{k_i}$ is the degree of a basic polynomial invariant over $\mathbb{R}$.
\item[(ii)] $|H_4|/p^{v_p(|H_4|)}$ divides the product of degrees $d_{i,p}$'s, where $v_p(m)$ denotes the maximal power of $p$ in $m$.
 (cf. \cite[(3.7)]{KP85}).
\end{itemize}

\begin{rem} We expect that following the arguments of~\cite{Ka85}, \cite{KP85} and \cite{NNS} one can verify the conditions (i) and (ii).
Namely, first one shows that the kernel of the characteristic map $\bar\rho_\tau \colon S_\tau \to \overline{\mathcal{Z}}(\mathcal{G}_\tau)$ modulo $p$ coincides 
with the ideal of generalized invariants of \cite[(0.2)]{KP85}, hence, generalizing \cite[(1)]{Ka85} to all real finite reflection groups (see also \cite[Thm.~2.1]{NNS}).
Then (i) and (ii) should follow from the arguments similar to those of \cite{KP85}.
\end{rem}

Taking the image of the presentation \eqref{eq:forme8} under $\hat\iota^*$ we obtain the following (conjectural)
answer for the virtual Chow ring of $H_4$ with coefficients in $l$:
\begin{equation}
\widehat{\mathcal{Z}}(H_4;l)\simeq \frac{l[y_3,y_4,y_5,y_{10},y_{15}]}
{(2y_3,3y_4,2y_5,3y_{10},2y_{15},y_3^4,y_4^3,y_5^4,y_{15}^2-y_{10}^3)},
\end{equation}
Indeed, polynomial invariants over $\mathbb{R}$
of the group $H_4$ of order $14400=2^6 3^25^2$ have degrees $2,12,20,30$ (see e.g. \cite[p.75]{Hi82}). Hence, 
2-generators of $\widehat{Z}(H_4;l)$ have degrees $3,5,15$ and exponents 4,4,2 and 
$3$-generators have degrees $4,10$ and exponents $3,3$ respectively.

Under the same assumptions, for the Lusztig foldings $D_6\to H_3$ and $A_4\to H_2$ we obtain the following
(conjectural) answers for the virtual Chow rings of $H_2$ and $H_3$ respectively:
\begin{equation}
\widehat{\mathcal{Z}}(H_2;l)\simeq \{1\}\quad \text{ and }\quad \widehat{\mathcal{Z}}(H_3;l)\simeq \frac{l[y_3,y_5]}{(2y_3,2y_5,y_3^2,y_5^2)}.
\end{equation}

\bibliographystyle{alpha}

\end{document}